\documentclass[a4paper,11pt]{article}

\usepackage{amsmath}
\usepackage{amsthm}
\usepackage{amsfonts}
\usepackage{amssymb}
\usepackage{url}
\usepackage{tikz}
\usepackage[bottom,symbol]{footmisc}
\usepackage{geometry}
\title{On applications of Razborov's flag algebra calculus to extremal 3-graph theory}
\author{Victor Falgas-Ravry \thanks{School of Mathematical Sciences, Queen Mary, University of London, Mile End Road, London E1 4NS, U.K. Email: {\tt v.falgas-ravry@qmul.ac.uk}.}
\and Emil R. Vaughan \thanks{School of Electronic Engineering and Computer Science, Queen Mary, University of London, Mile End Road, London E1 4NS, U.K. Email: {\tt e.vaughan@qmul.ac.uk}. Supported by EPSRC grant EP/H016015/1.}}

\newtheorem{theorem}{Theorem}
\newtheorem{proposition}[theorem]{Proposition}

\newtheorem{lemma}[theorem]{Lemma}
\newtheorem{remark}{Remark}
\newtheorem{conjecture}{Conjecture}
\newtheorem{question}[conjecture]{Question}
\newtheorem{observation}[remark]{Observation}
\theoremstyle{definition}
\newtheorem{construction}{Construction}

\DeclareMathOperator{\ex}{ex}
\DeclareMathOperator{\E}{\mathbb{E}}

\providecommand{\size}[1]{\left|#1\right|}

\providecommand{\flagmatic}{Flagmatic}
\newcommand{\Text}[1]{\text{\textnormal{#1}}}

\newcommand{\website}{\url{http://maths.qmul.ac.uk/~ev/flagmatic}}

\newcommand{\sinsixty}{0.8660254} 
\newcommand{\midfano}{2.598076} 

\begin{document}

\maketitle
\begin{abstract}
In this paper, we prove several new Tur\'an density results for 3-graphs with independent neighbourhoods. We show:
\[\pi(K_4^-, C_5, F_{3,2})=12/49, \ \ \pi(K_4^-, F_{3,2})=5/18 \ \ \text{and} \]
\[\pi(J_4, F_{3,2})=\pi(J_5, F_{3,2})=3/8, \]
where $J_t$ is the 3-graph consisting of a single vertex $x$ together with a disjoint set $A$ of size $t$ and all $\binom{\size{A}}{2}$ 3-edges containing $x$.
We also prove two Tur\'an density results where we forbid certain induced subgraphs:
\[ \pi(F_{3,2}, \Text{ induced } K_4^-)=3/8 \ \ \text{and} \]
\[ \pi(K_5,  \Text{ 5-set spanning 8 edges})=3/4. \]
The latter result is an analogue for $K_5$ of Razborov's result that 
\[\pi(K_4,  \Text{ 4-set spanning 1 edge})=5/9.\]

We give several new constructions, conjectures and bounds for Tur\'an densities of 3-graphs which should be of interest to researchers in the area. Our main tool is `\flagmatic', an implementation of Razborov's flag algebra calculus, which we are making publicly available. In a bid to make the power of Razborov's method more widely accessible, we have tried to make 
\flagmatic\ as user-friendly as possible, hoping to remove thereby the major hurdle that needs to be cleared before using the flag algebra calculus.

Finally, we spend some time reflecting on the limitations of our approach, and in particular on which problems we may be unable to solve. Our discussion of the `complexity barrier' for the flag algebra calculus may be of general interest.
\end{abstract}

\section{Introduction} 
Extremal graph and hypergraph theory have in recent years seen a string of results obtained by application of the flag algebra calculus developed by Razborov~\cite{Razborov1}. With the notable exception of the Fano plane, most known Tur\'an density results for 3-graphs have been obtained anew using his method, as well as some new results and the best known upper bounds for several other problems~\cite{Razborov2, Razborov3}. Particularly impressive in this respect was Razborov's proof of Tur\'an's conjecture under an additional restriction~\cite{Razborov2}:
\[\pi(K_4, \Text{ 4-set spanning 1 edge})=5/9.\]

In this paper, we use the flag algebra calculus to prove several new Tur\'an density results. In Section 3.1 we develop the extremal theory of 3-graphs with independent neighbourhoods, proving:
\[\pi(K_4^-, C_5, F_{3,2})=12/49,\]
\[\pi(K_4^-, F_{3,2})=5/18 \ \ \text{and} \]
\[\pi(J_4, F_{3,2})=\pi(J_5, F_{3,2})=3/8,\]
where $J_t$ is the 3-graph consisting of a vertex $x$ together with a disjoint set $A$ of size $t$ and all $\binom{\size{A}}{2}$ 3-edges containing $x$.
In Section 3.2, we prove two density results where we forbid certain induced subgraphs:
\[ \pi(F_{3,2}, \text{ induced } K_4^-)=3/8 \ \ \text{and} \]
\[ \pi(K_5,  \text{ 5-set spanning 8 edges})=3/4. \]
The latter result is an analogue for $K_5$ of the aforementioned theorem of Razborov for $K_4$. In addition we provide a number of new bounds, constructions and conjectures which may be of general interest.

Our main tool is \flagmatic, an implementation of the flag algebra calculus, which was written by the second author, and which we are making publicly available. Razborov's flag algebra calculus is an efficient formalism for computing density bounds in extremal combinatorics. In the case of extremal 3-graph theory, it does this by reducing an initial problem of proving inequalities for subgraph densities to a  semi-definite programming problem, which in some cases can be solved exactly with the aid of a computer. We discuss what `in some cases' means in greater detail in Section~4. Let us only say for the moment that without extra ideas we cannot hope for a general extremal theory to emerge from a direct application of Razborov's flag algebra calculus.

However, given the difficulty of extremal 3-graph theory and the paucity of known results, an implementation of the flag algebra calculus such as \flagmatic\ can be of great help to theory-building efforts, by providing many useful bounds and guiding investigations towards attainable goals. A major hurdle for mathematicians wishing to use the flag algebra calculus in their work is the need of a computer program to assist them in the calculations. \flagmatic\ was designed with this in mind, and we have tried to make it as user-friendly as possible.

As the flag algebra calculations involved in our proofs are very long and not terribly informative, we have produced `certificates' rather than write them out in full. The certificates are available on the \flagmatic\ website
\begin{quote}
\website
\end{quote}
where the interested reader may also download a copy of \flagmatic\ for herself. In addition, our results have also been independently verified by Baber and Talbot~\cite{Baber}.

We should stress that our proofs are computer assisted rather than computer generated; indeed in every case we could produce `proofs by hand' by doing a lot of enumeration and computations, and then pulling some very large positive semi-definite matrices out of our hat. This would take thousands of pages however, and would not be very informative. We have therefore opted not to do so.

This paper is structured as follows: after introducing a small amount of notation, Section 2 is devoted to explaining how \flagmatic\ works, beginning with an exposition of the flag algebra calculus (Section 2.2), some remarks about \flagmatic\ (Section 2.3), and a discussion of the proof certificates it produces (Section 2.4).

Section 3 contains our main results. In Section 3.1 we develop an extremal theory of 3-graphs with independent neighbourhoods, proving the first set of results mentioned in the introduction and providing several new constructions and conjectures; in Section 3.2 we consider forbidding induced subgraphs, obtaining in particular a theorem related to the conjecture of Tur\'an that $\pi(K_5)=3/4$; in Section 3.3 we go on to discuss `non-principality' (the fact that $\pi(\mathcal{F}\cup \mathcal{G}) < \min \left(\pi(\mathcal{F}), \pi (\mathcal{G}) \right)$ for some families of 3-graphs $\mathcal{F}, \mathcal{G}$).

Finally in Section 4 we consider the limits inherent to our approach, in particular the `complexity barrier' it runs into. We end with some open questions and a summary of results and constructions.

\section{The flag algebra calculus}
\subsection{Some notation and definitions}
We begin with some notation and definitions, most of which are standard. A \emph{3-graph} $G$ is a pair of sets $G=(V,E)$, with $V=V(G)$ a set of vertices, and $E=E(G)$ a collection of 3-sets from $V$, which are the \emph{3-edges} of $G$. Given a family of 3-graphs $\mathcal{F}$, we say that a 3-graph $G$ is \emph{$\mathcal{F}$-free} if $G$ contains no member of $\mathcal{F}$ as a subgraph. We write $\ex(n, \mathcal{F})$ for the maximal number of 3-edges that can be present in an $\mathcal{F}$-free 3-graph. The nonnegative function $\ex(n, \mathcal{F})$  is referred to as the \emph{Tur\'an number} of $\mathcal{F}$.

An easy averaging argument shows that $\ex(n, \mathcal{F})/\binom{n}{3}$ is nonincreasing and hence tends to a limit as $n\rightarrow \infty$. This limit, denoted by $\pi(\mathcal{F})$, is the \emph{Tur\'an density} of $\mathcal{F}$. It is the asymptotically maximal proportion of edges present in an $\mathcal{F}$-free 3-graph. The standard \emph{Tur\'an (density) problem} for 3-graphs is: given a family $\mathcal{F}$, determine $\pi(\mathcal{F})$. The analogous question for 2-graphs has been completely answered by the Erd\H os-Stone Theorem; by contrast very few Tur\'an densities of 3-graphs are known. (See the recent survey paper of Keevash~\cite{Keevash} for details.)

We say that a particular instance of the Tur\'an problem for 3-graphs is \emph{stable} if there is a sequence of 3-graphs
\[G_1, \ G_2, \ \dots, \ G_n, \ \dots \]
such that for any $\varepsilon>0$ there exists $\delta>0$ and $n_0 \in \mathbb{N}$ such that any $\mathcal{F}$-free 3-graph on $n \ge n_0$ vertices with more than $(\pi(\mathcal{F})-\delta)\binom{n}{3}$ 3-edges can be transformed into $G_n$ by adding or deleting fewer than $\varepsilon n^3$ 3-edges. (Intuitively, this says there is an essentially unique extremal configuration, and that any `close to extremal' 3-graph must lie at a small `edit' distance from it.)

Let us now define various standard 3-graphs that appear in this paper. We shall write $[n]$ for $\{1,2, \dots n\}$, and when enumerating 3-edges, we shall often write $xyz$ for $\{x,y,z\}$. When there is no confusion possible, we may also use  `edge' for `3-edge', `graph' for `3-graph' and `subgraph' for `3-subgraph'. Given a set $A$ and an integer $r$, we shall write $A^{(r)}$ for the set of $r$-sets of $A$.

The \emph{complete} 3-graph on $t$ vertices is the 3-graph $K_t=([t], [t]^{(3)})$. Deleting a single 3-edge from $K_4$ yields a copy of $K_4^-$, the unique (up to isomorphism) 3-graph on $4$ vertices with 3 edges. We let $C_5$ denote the \emph{(strong) 5-cycle} $C_5=([5], \{123,234,345, 451, 512\})$.

We shall also touch on \emph{links}. Given a 3-graph $G$ and $x \in V(G)$, the \emph{link graph} (or \emph{link}) of $x$ in $G$ is the 2-graph
\[G_x=\left(V\setminus \{x\}, \{ab: \ xab\in E(G)\}\right).\] We shall consider the problem of forbidding the links of a 3-graph from containing a complete 2-graph on $t$ vertices, and we define $J_t$ to be the corresponding forbidden 3-subgraph, namely \[J_t=\left([t+1], \left\{\{x,y,t+1\}:\ \{xy\}\in[t]^{(2)} \right\}\right).\] This 3-graph $J_t$ is a special case of a `suspension' (namely the 3-suspension of $K_t^2$); in the more general notation due to Keevash~\cite{Keevash} it is denoted by $S^3K_t^2$.

Various constructions we consider in this paper involve taking a (possibly unbalanced) partition of the vertex set $V= A_1 \sqcup A_2 \sqcup \dots \sqcup A_r$ and then adding edges according to some rule. In this setting, a 3-edge has type $A_iA_j A_k$ if it is of the form $xyz$ with $x \in A_i, y \in A_j, z \in A_k$.

A \emph{blow-up} construction is obtained by taking a $3$-graph $H$ on $V(H)=[r]$ with some possibly degenerate edges---for example `112' or `333'---and using it as a template to construct configurations for graphs of order $n$ for every $n \in \mathbb{N}$ as follows: \begin{itemize}
\item partition $[n]$ into $r$ parts $A_1 \sqcup A_2 \sqcup \dots \sqcup A_r$
\item add all edges of type $A_iA_jA_k$ with $ijk \in E(H)$
\end{itemize}
An \emph{iterated blow-up} construction is obtained, as the name suggests, by taking a blow-up construction from a template $H$ and then repeating the construction inside (some of) the $\size{V(H)}$ parts of the resulting 3-graph, and then again in the resulting subparts, and so on. The partition and edges obtained by the first iteration are said to be at \emph{level 1} of the construction, the subpartition and edges given by the second iteration are said to lie at \emph{level 2}, and so on.

Finally and most importantly, given a 3-graph $G$ of order $\size{V(G)}=n$ and a 3-graph $H$ of order $m\leq n$, let us define the \emph{(induced) subgraph density} of $H$ in $G$, denoted by $d_H(G)$ to be the probability that an m-subset of $V(G)$ chosen uniformly at random induces a copy of $H$ in $G$, i.e. that the resulting random subgraph of G is isomorphic to $H$. When $H$ is the 3-edge $([3], \{123\})$, we write $d(G)$ for $d_H(G)$ and call it the \emph{(edge) density} of $G$.

\subsection{Mantel's theorem via the flag algebra calculus}
For the sake of making this paper self-contained, we shall give here a brief overview of the flag algebra calculus. As stated in the introduction, it consists of an efficient formalism introduced by Razborov~\cite{Razborov1} for converting the problem of proving certain inequalities between subgraph densities into a semi-definite programming problem, which can be solved with the aid of a computer. Excellent expositions of this calculus from an extremal combinatorics perspective have already appeared in the literature; our presentation draws in particular on Section 7 of~\cite{Keevash} and Section 2.1 of~\cite{BaberTalbot}.

For ease of notation and the sake of clarity, we shall consider 2-graphs rather than 3-graphs for our exposition, in contrast to~\cite{BaberTalbot, Keevash}. Razborov~\cite{Razborov1} in fact defined his flag algebra calculus in a much more general setting which includes 2-graphs and 3-graphs as special cases; we feel that the 2-graph case gives all the intuition necessary, while keeping calculations to a minimum.

Let $K_3^{(2)}$ denote the complete 2-graph on 3 vertices, otherwise known as the triangle. To illustrate our discussion, we shall use the following weak form of Mantel's Theorem as a running example:

\begin{theorem}\label{weak mantel}
\[\pi(K_3^{(2)})=1/2.\]
\end{theorem}

What would be the crudest possible way of finding a nontrivial upper bound on $\pi(K_3^{(2)})$? We could observe that
a triangle-free graph $G$ on $n$ vertices is at most as dense as the most dense subgraph of order $m\leq n$ that it contains. Note that as $G$ is triangle-free, so are its subgraphs. Say therefore that a subgraph is \emph{admissible} if it is triangle-free and so could occur as a subgraph of $G$. Pick some integer $m$, and let $\mathcal{H}$ denote the collection of all \emph{admissible subgraphs} of order $m$ up to isomorphism. We then have
\begin{equation}\label{eq1}
d(G)= \sum_{H \in \mathcal{H}} d_H(G) d(H)
\end{equation}
with $\sum_{H \in \mathcal{H}} d_H(G)=1$, and thus 
\begin{equation}\label{eq2}
d(G)\leq \max_{H \in \mathcal{H}} d(H).
\end{equation}

This is fairly obviously a poor way to go about bounding $\pi(K_3^{(2)})$. Indeed pick for example $m=3$. The family $\mathcal{H}$ then consists of three graphs $H_0,H_1, H_2$, with $H_i$ being the unique (up to isomorphism) graph on 3 vertices with exactly $i$ edges. Thus (\ref{eq2}) shows $\pi(K_3^{(2)})\leq 2/3$, but this could only be sharp if \emph{all} induced subgraphs of order $3$ were isomorphic to $H_2$. This is impossible for $n\geq 5$. Indeed,
 suppose we have $x, y$ with $xy$ a non-edge, and $a,b,c$ such that $\{xya\},\{xyb\}$ and $\{xyc\}$ all induce copies of $H_2$ in $G$. Then as $G$ is triangle-free, $\{abc\}$ must induce a copy of $H_0$. We therefore expect the density of $H_0$ in $G$ to be bounded below by some function of $H_2$ (the density of $H_1$ being determined by the fact that $\sum_i d_{H_i}(G)=1$). Thus one way we could try to refine inequality~(\ref{eq2}) would be to take such a relationship and exploit it to improve our bound.

The simplest relationship of this kind we could hope for is a linear inequality for subgraph densities of the form
\[\sum_{H \in \mathcal{H}} d_{H}(G) a_H \geq 0.\]
Given such an inequality, inequality~(\ref{eq2}) can be changed to 
\[d(G) \leq \max_{H \in \mathcal{H}}\left(d(H)+a_H\right).\] 
Provided our linear inequality is `good', the $a_H$ `even out' the coefficients $d(H)+a_H$ by transferring weight from dense subgraphs to sparser ones, improving on~(\ref{eq2}).

Following this line of thought, we then ask ourselves: how can we produce (good) linear inequalities for subgraph densities? Our remark on the fact that we cannot pack a graph full of induced copies of $H_2$ suggests a possible answer: we can consider the ways in which different kinds of subgraphs can intersect, and from this information derive bounds on subgraph densities. What Razborov's flag algebra calculus gives us is an efficient formalism for doing just that, which we now present.

Suppose we work in the general framework of $\mathcal{F}$-free graphs. (Our example had $\mathcal{F}=\{K_3^{(2)}\}$.) Let $m$ be an integer, which we shall fix later on, and let $\mathcal{H}$ denote as before the set of all (up to isomorphism) admissible subgraphs of order $m$.

An \emph{intersection type} is a graph on a labelled vertex set, with every vertex having a distinct label. Given an intersection type $\sigma$, a \emph{$\sigma$-flag} is an admissible graph $F$ on a partially labelled vertex set such that the subgraph induced by the labelled vertices is a copy of $\sigma$ (with identical labels for the vertices.) For example, let us consider the intersection type $\sigma$ consisting of a single vertex labelled `1'. Then there are (up to isomorphism) two $\sigma$-flags of order $2$, namely $F_0$ consisting of a non-edge with one end labelled `1', and $F_1$ consisting of an edge with one end labelled `1' (see Figure \ref{typeandflags}.) We shall write $\mathcal{F}_{\sigma}^l$ for the collection of all (up to isomorphism) $\sigma$-flags of order $l$.

\begin{figure}
\begin{center}
\begin{tikzpicture}
\draw[black,fill=black] (0,0) circle (2pt) (1,0) circle (2pt) (0.5,\sinsixty) circle(2pt)
(3,0) circle (2pt) -- (4,0) circle (2pt) (3.5,\sinsixty) circle(2pt)
(6,0) circle (2pt) -- (6.5,\sinsixty) circle(2pt) -- (7,0) circle (2pt);
\draw (0.5,-0.5) node {$H_0$}
(3.5,-0.5) node {$H_1$}
(6.5,-0.5) node {$H_2$};
\end{tikzpicture}
\end{center}
\caption{The admissible graphs.}
\label{admissiblegraphs}
\end{figure}
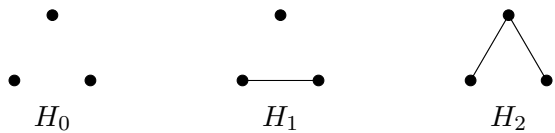

Let us now define some flag densities. Fix an intersection type $\sigma$ of order $\size{V(\sigma)}=s$ and an integer $l\geq s$. Given a graph $G$, select a partial labelling of $V(G)$ with the labels from $\sigma$, chosen uniformly at random. (By which we mean: randomly select $\size{V(\sigma)}$ vertices and assign them distinct labels from $\sigma$.) This makes $G$ into a potential $\sigma$-flag. Note that the labelled vertices could fail to induce a copy of $\sigma$, and that we allow this. Now select a set $S_1$ of $l-s$ other vertices (necessarily unlabelled) uniformly at random. Taken together with the labelled vertices, $S_1$ gives us a potential $\sigma$-flag of order $(l-s)+s=l$; so, given $F \in \mathcal{F}_{\sigma}^l$, write $d_{F}(G)$ for the probability this is a copy of $F$. We call this the \emph{flag density} of $F$ in $G$.

Having selected $S_1$, pick a disjoint set $S_2$ of $l-s$ unlabelled vertices uniformly at random. Taken together with the partially labelled vertices, $S_1$ and $S_2$ give us two potential $\sigma$-flags of order $l$. Then, given $F,F' \in \mathcal{F}_{\sigma}^l$, let $d_{F,F'}(G)$ be the probability $S_1$ and $S_2$ induce copies of $F$ and $F'$ respectively when taken together with the labelled vertices. We call $d_{F,F'}(G)$ the \emph{flag pair density} of $(F,F')$ in $G$.

Finally, for a fixed partial labelling $\theta$ of $V(G)$ with labels from $\sigma$, select an $(l-s)$-set $S_1$ from the unlabelled vertices of $G$ uniformly at random, and write $d^{\,\theta}_{F}(G)$ for the probability $S_1$ together with the vertices labelled by $\theta$ induces a copy of the the $\sigma$-flag $F$. Then select a disjoint $(l-s)$-set $S_2$ uniformly at random from the remaining unlabelled vertices and write $d^{\,\theta}_{F,F'}(G)$ for the probability $S_1$ and $S_2$ induce copies of $F$ and $F'$ respectively when taken together with the vertices labelled by $\theta$.

In our running example with $\sigma$ consisting of a single vertex labelled `1', $d_{F_1}(G)$ measures the probability that if we randomly label a vertex $x$ in $V(G)$ and randomly select $y \in V(G) \setminus \{x\}$ then $xy \in E(G)$---in other words, $d_{F_1}(G)$ is exactly the edge-density of $G$. On the other hand, $d_{F_1, F_0}(G)$ measures something slightly more complicated: letting $n=\size{V(G)}$ and writing $d(v)$ for the degree of $v$ in $G$, we have
\[d_{F_1,F_0}(G)= \sum_{v \in V(G)}\frac{1}{n}\left(\frac{d(v)}{n-1}\right) \left(\frac{n-1 -d(v)}{n-2}\right).\]
More interesting from a combinatorial perspective is
\[d_{F_1,F_1}(G)= d_{H_2}(G)/3 +d_{K_3^{(2)}}(G),\]
which in a triangle-free graph measures the $H_2$ density (divided by $3$). 

Now, let us fix $\sigma, l$ and make two easy observations. Firstly, if $n=\size{V(G)}$ is sufficiently large, then picking two random extensions of order $l-s$ for a randomly labelled set of $s$ vertices is essentially the same thing as picking a random pair of \emph{disjoint} extensions---indeed the probability that two randomly chosen $(l-s)$-sets from $V(G)$ intersect is $O(1/n)$.

\begin{observation}\label{indep}
For all $F,F' \in \mathcal{F}_{\sigma}^l$, for all possible partial labellings $\theta$ of $V(G)$ with labels from $\sigma$,
\[d^{\,\theta}_F(G) d^{\,\theta}_{F'}(G) = d^{\,\theta}_{F,F'}(G) +O\left(1/n\right).\]
In particular, taking expectations over $\theta$ on both sides, we have
\[\E_{\theta} d^{\,\theta}_{F}(G)d^{\,\theta}_{F'}(G) = d_{F,F'}(G) +O\left(1/n\right).\]
\end{observation}

Secondly, we can average:

\begin{observation}\label{average}
Let $m$ be any integer with $m\geq 2l-s$, and let $\mathcal{H}$ be the family of all (up to isomorphism) admissible subgraphs of order $m$ defined earlier. Then for all $F,F' \in \mathcal{F}_{\sigma}^l$,
\[d_{F,F'}(G)= \sum_{H \in \mathcal{H}} d_H(G) d_{F,F'}(H).\]
\end{observation}

The appearance of the $d_H(G)$ terms in Observation~\ref{average} suggests we are close to achieving our goal. And indeed, let $Q$ be any fixed positive semi-definite $\vert \mathcal{F}_{\sigma}^l\vert \times \vert \mathcal{F}_{\sigma}^l\vert$ matrix with entries indexed by $\mathcal{F}_{\sigma}^l$. Then
\begin{align}
0 &\leq  \E_{\theta} \sum_{F,F' \in \mathcal{F}_{\sigma}^l} Q_{F,F'} d^{\,\theta}_{F}(G) d^{\,\theta}_{F'}(G) && \text{(by positive semi-definiteness)}\notag \\
& = \sum_{F,F' \in \mathcal{F}_{\sigma}^l}Q_{F,F'} d_{F,F'}(G) +O(1/n) && \text{(by Observation~\ref{indep})} \notag \\
& = \sum_{F,F' \in \mathcal{F}_{\sigma}^l}Q_{F,F'} \sum_{H \in \mathcal{H}} d_H(G) d_{F,F'}(H) + O(1/n) && \text{(by Observation~\ref{average})} \notag \\
&=\sum_{H \in \mathcal{H}} d_H(G) \left(\sum_{F,F' \in \mathcal{F}_{\sigma}^l} Q_{F,F'} d_{F,F'}(H)\right) +O(1/n) \label{singletype} \end{align}
by changing order of summation again in the last line.

This is of the desired form $0 \leq \sum_{H \in \mathcal{H}} d_H(G) \lambda_H +O(1/n)$ (the $O(1/n)$ error term being irrelevant when bounding the Tur\'an density). Thus for a fixed $m$, every choice of $\sigma$ and $l$ such that $2l- \size{V(\sigma)} \leq m$, and positive semi-definite matrix $Q$, gives us some linear inequality between subgraph densities for admissible subgraphs of order $m$. We can then sum these inequalities together. For example, if we have $r$ choices
\[(\sigma_1, l_1, Q_1), (\sigma_2, l_2, Q_2), \dots, (\sigma_r, l_r, Q_r),\]
we can add the corresponding inequalities~(\ref{singletype}) to get
\[0 \leq \sum_{H \in \mathcal{H}} d_H(G) \left(\sum_{i=1}^r\sum_{F,F' \in 
\mathcal{F}(\sigma_i,l_i)} \left(Q_i\right)_{F,F'} d_{F,F'}(H)\right) +O(1/n).\]
With a view to getting the best possible improvement of~(\ref{eq2}), we can, for a fixed choice of $(\sigma_1, l_1)$, $(\sigma_2, l_2)$, $\dots$, $(\sigma_r, l_r)$, optimise the choice of the matrices $Q_1, Q_2, \dots , Q_r$ to obtain a `best inequality possible': 
\[0 \leq \sum_{H \in \mathcal{H}}d_H(G) a_H +O\left(1/n\right),\]
where
\[a_H= \sum_{i=1}^r\sum_{F,F' \in \mathcal{F}(\sigma_i,l_i)} \left(Q_i\right)_{F,F'} d_{F,F'}(H). \]
We can add this to (\ref{eq1}) to get
\begin{equation} d(G_n) \leq \sum_{H \in \mathcal{H}} d_H(G_n) \left(d(H)+a_H\right)+O\left(1/n\right), \label{prefab} \end{equation}
and thus obtain a bound on the Tur\'an density of our family $\mathcal{F}$ of forbidden subgraphs, 
\begin{equation} \pi(\mathcal{F}) \leq \max_{H \in \mathcal{H}} \left(d(H)+a_H\right). \label{fab} \end{equation}
We refer to (\ref{fab}) as the \emph{flag algebra bound}, and for each $H \in \mathcal{H}$ we call $d(H)+a_H$ the \emph{flag algebra coefficient} of $H$ in the bound.

At this point, let us make two important observations:

\begin{lemma}\label{tightcoeff}
Suppose the flag algebra bound is tight, i.e{.}
\[\pi(\mathcal{F})=\max_{H} \left(d(G)+a_H\right),\]
and there is an admissible subgraph $H'$ whose flag algebra coefficient is $\rho$ where $\rho < \pi(\mathcal{F})$. Then, for any sequence of $\mathcal{F}$-free graphs $(G_n)_{n \in \mathbb{N}}$ with $\size{V(G_n)}=n$ and $e(G_n)= \left(\pi(\mathcal{F})+o(1)\right)\binom{n}{2}$, we have
\[\limsup_{n \rightarrow \infty} d_{H'}(G_n)=0.\]
\end{lemma}

\begin{proof}
Indeed, suppose $\limsup_{n \rightarrow \infty} d_{H'}(G_n)>\varepsilon$ for some fixed $\varepsilon>0$. Then by (\ref{prefab}), we have
\[ d(G_n) < \varepsilon \rho + (1 - \varepsilon)\pi(\mathcal{F}) + O(1/n)\]
which is less than $\pi(\mathcal{F})$ for $n$ large enough, a contradiction.

\end{proof}

Similarly, a consequence of requiring the flag algebra bound to be tight is that
\begin{equation} \sum_{F,F' \in \mathcal{F}_{\sigma}^l} \E_{\theta} Q_{F,F'} d^{\,\theta}_{F}(G) d^{\,\theta}_{F'}(G)=O(1/n)\label{limevec} \end{equation}
for all our optimised choices of $(\sigma, l, Q)$ and graphs $G$ that are `close' to being extremal.

Additionally, if the problem is stable with a blow-up or iterated blow-up construction of some finite $3$-graph being best possible, let us consider for a moment the `limit'  of a sequence of extremal configurations $G_n$ as $n\rightarrow \infty$. For all $F \in \mathcal{F}_{\sigma}^l$, the quantity $d_{F}^{\,\theta}(G_n)$ is determined (up to $o(1)$) by the parts in which we set the labelled vertices; in particular if $\theta$ and $\theta'$ place the same labels in the same parts, then  $d_{F}^{\,\theta}(G_n)=d_{F}^{\,\theta'}(G_n)+o(1)$ for all choices of $F$, and we may treat $\theta$ and $\theta'$ as being `equivalent'. We can reduce in this way the set of all partial labelings into a finite set of `equivalence' classes.

To illustrate this informal discussion with an example, suppose the extremal configurations $G_n$ consist of complete balanced bipartite graphs and that $\vert V(\sigma) \vert=2$. Then there are two `equivalence' classes of partial labelings: one in which both labelled vertices are put in the same part of $G_n$, and one in which the labelled vertices are assigned to different parts of $G_n$.

Now for each `equivalence' class, choose a sequence of distinct representatives, i.e{.} a sequence of partial labelings of $n$-vertex extremal configurations, and write $U$ for the finite set of sequences thus defined. Since $Q$ is positive semi-definite, we have the following:

\begin{remark}[Baber~\cite{BaberThesis}] \label{zeroevec}
Suppose the flag algebra bound is tight, and that the problem admits a blow-up or an iterated blow-up as the stable extremal configuration. Let $(\sigma, l, Q)$ be one of our optimised choices of intersection type, flag order and matrix. Then for any $(\theta_n)_{n \in \mathbb{N}} \in U$, where $U$ is the set of sequences of partial labelings informally defined above,
\[\lim_{n \rightarrow \infty} \sum_{F,F' \in \mathcal{F}_{\sigma}^l}Q_{F,F'}d^{\,\theta_n}_{F}(G_n)d^{\,\theta_n}_{F'}(G_n)=0.\]
\end{remark}

In other words, the vectors of flag densities associated with a fixed embedding of $\sigma$ in a large extremal configuration accumulate around the set consisting of the zero vector and of the zero eigenvectors of the positive semi-definite matrix $Q$. This remark was first made in a more formal infinitary setting by Baber~(Lemma 2.4.4 in~\cite{BaberThesis}), to whom we refer the reader for a rigorous proof.

Having made these two observations, let us return to our running example as an illustration of how Razborov's method is used to provide upper bounds for Tur\'an densities. Recall that we are trying to show $\pi(K_3^{(2)}) \leq 1/2 $ using the flag algebra calculus. In this case consideration of one intersection type suffices, namely the type $\sigma$ consisting of a single labelled vertex. We have two $\sigma$-flags of order $2$, $F_0$ and $F_1$, and three admissible subgraphs of order $3$, $H_0, H_1$ and $H_2$ (see Figures \ref{admissiblegraphs} and \ref{typeandflags}). Let us compute $d_{F,F'}(H)$ for all possible choices of $F,F'$ and $H$.

\begin{figure}
\begin{center}
\begin{tikzpicture}
\draw[black,fill=black] (0,0) circle (2pt) node [above] {$1$}
(2,0) circle (2pt) node [above] {$1$} (3,0) circle (2pt)
(5,0) circle (2pt) node [above] {$1$} -- (6,0) circle (2pt);
\draw (0,-0.5) node {$\sigma$}
(2.5,-0.5) node {$F_0$}
(5.5,-0.5) node {$F_1$};
\end{tikzpicture}
\end{center}
\caption{The intersection type $\sigma$ and $\sigma$-flags $F_0$ and $F_1$.}
\label{typeandflags}
\end{figure}
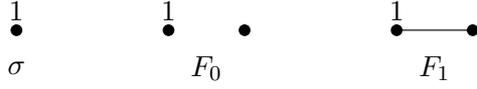

Since our intersection type $\sigma$ consists of a single vertex, our random labelling and our two random extensions always give us an ordered pair of $\sigma$-flags, so that $\sum_{F,F'} d_{F,F'}(H)=1$. Now it is easy to see that $d_{F_0, F_0}(H_0)=1$, and that $d_{F,F'}(H_0)=0$ for all other choices of $F,F'$. Next, we see that $d_{F_0, F_0}(H_1)=1/3$, as the only way of getting two copies of $F_0$ is to label the unique degree zero vertex in $H_1$ `1' (which happens a third of the time), and that with this labelling we always get two copies of $F_0$ in the randomly chosen extensions. Also $d_{F_1,F_1}(H_1)=0$ as $H_1$ contains only one edge, so that we have by symmetry $d_{F_0,F_1}(H_1)=d_{F_1,F_0}(H_1)=1/3$. We then get the $d_{F,F'}(H_2)$ for free by noting that $H_2$ is the complement of $H_1$ and $F_0$ is the complement of $F_1$, so that $d_{F_{\epsilon}, F_{\eta}}(H_2)= d_{F_{1-\epsilon}, F_{1-\eta}}(H_1)$, but we would encourage the reader to calculate these directly for herself instead. Summarising, we have:

\vspace{5mm}
\begin{center}
\begin{tabular}{cccc}
& \ $d_{F_0,F_0}(H)$ \ & \ $d_{F_0, F_1}(H)=d_{F_1,F_0}(H)$ \ & \ $ d_{F_1, F_1}(H)$ \ \\ \hline
$H_0$ & $1$ & $0$ & $0$ \\
$H_1$ & $1/3$ & $1/3$ & $0$ \\
$H_2$ & $0$ & $1/3$ & $1/3$ \\ \hline
\end{tabular}
\end{center}
\vspace{5mm}

Now let 
\[Q=\left( \begin{matrix} a & b \\ c & d \end{matrix}\right)\]
 be a positive semi-definite matrix. (In other words $a,b,c,d$ satisfy $a \geq 0, \ ad-bc\geq 0$.) 
Then in any triangle-free graph $G$ of order $n$, 
\[ d(G)\leq d_{H_0}(G)\left(0+a_{H_0}\right)+d_{H_1}(G)\left(\frac{1}{3}+a_{H_1}\right)+d_{H_2}(G)\left(\frac{2}{3}+a_{H_2}\right) +O(1/n),\]
where the $a_{H_i}$ are the coefficients introduced earlier, given by
\begin{align*}
a_{H_0} & = a\\
a_{H_1} & = a/3 + b/3 + c/3\\
a_{H_2} & = b/3 + c/3 + d/3.
\end{align*}

We now optimise the choice of $Q$. Guessing that extremal triangle-free graphs are complete bipartite, we expect by Lemma~\ref{tightcoeff} that both $H_0$ and $H_2$ must both have flag algebra coefficients equal to $1/2$ in a tight flag algebra bound; it is then a straightforward exercise in calculus to work out that
\[Q= \left( \begin{matrix} 1/2 & -1/2 \\ -1/2 & 1/2 \end{matrix}\right)\]
is an optimal choice of matrix.

Our optimal inequality is then
\[0 \leq \frac{d_{H_0}(G)}{2} -\frac{d_{H_1}(G)}{6}- \frac{d_{H_2}(G)}{6} +O\left(1/n\right),\]
 giving
\[d(G) \leq \frac{1}{2}\left(d_{H_0}(G)+d_{H_2}(G)\right) + \frac{1}{6}d_{H_1}(G) +O\left(1/n\right).\] 
Taking the limit as $n\rightarrow \infty$, we deduce that $\pi(K_3^{(2)})\leq 1/2$. Since a complete balanced bipartite graph achieves density $1/2+o(1)$, we in fact must have equality. We have thus proved Theorem~\ref{weak mantel}. (In fact we have proved a little more: our inequality tells us exactly which subgraphs can have positive density in an extremal example, and what those positive densities are, namely $d_{H_0}(G)=1/4+o(1)$ and $d_{H_2}(G)=3/4+o(1)$. This information can then be used to show that `close' to extremal triangle-free graphs are `close' to complete bipartite. However this goes beyond the scope of this exposition.)

In general it is not practical to do the optimisation above by hand (or indeed to perform manually all of the earlier calculations required to determine $\mathcal{H}, \mathcal{F}_{\sigma}^l$ and the $d_{F,F'}(H)$ terms), and this is where \flagmatic\ comes in: taking as input a set of forbidden configurations $\mathcal{F}$ and an integer $m$, it performs all the required computations, feeds the problem in an appropriate form into a semi-definite problem solver (SDP solver) then converts the SDP solver output into a bound on $\pi(\mathcal{F})$ and produces a `certificate' of the flag algebra calculation. We discuss all this in detail in the next subsection.

\subsection{Flagmatic}

All the upper bounds on Tur\'an densities that we give in this paper have been obtained by flag algebra calculations assisted by \flagmatic, the program written by the second author to implement the flag algebra method. In this subsection we make some remarks concerning \flagmatic, and, in particular, how it obtains exact solutions. Note that in the remainder of the paper, starting from this section, we shall write `graph' for `3-graph'.

\flagmatic\ takes as input a family of forbidden graphs $\mathcal{F}$, and an integer $m$. It then determines $\mathcal{H}$, the family of all admissible ($\mathcal{F}$-free) graphs of order $m$, up to isomorphism, and generates a set of intersection types and flags to use. By default, \flagmatic\ will use all intersection types $\sigma$ whose order is congruent to $m$ modulo 2. For each $\sigma$, \flagmatic\ takes $\mathcal{F}_{\sigma}^l$ with $l = \left(m - \size{V(\sigma)}\right)/2$ as its family of $\sigma$-flags. \flagmatic\ then computes the densities $d(H)$, for each $H \in \mathcal{H}$, and all the flag pair densities $d_{F,F'}(H)$ for all $H \in \mathcal{H}$ and all pairs $F,F' \in \mathcal{F}_{\sigma}^l$.

(It is not hard to show that if we use a type $\sigma$ of order $s$, where $s$ is not congruent to $m$ modulo 2, then we can achieve at least as good a bound by replacing $\sigma$ with all the types of order $s+1$ that contain $\sigma$ as a labelled subgraph. For this reason, if we include all types whose order is congruent to $m$ modulo 2, then the bound we get will be no worse than if we use all the types.)

\flagmatic\ uses the semi-definite program (SDP) solver `CSDP' \cite{csdp} to find symmetric matrices $Q_1, Q_2, \dots, Q_r$ that optimise the flag algebra bound (\ref{fab}). (Note that in the search for optimal matrices, we may assume that each $Q_i$ is symmetric, for otherwise we could replace $Q_i$ by $(Q_i + Q_i^T)/2$ without changing $a_H$.) As is standard for this kind of software, CSDP uses floating-point arithmetic, which presents us with a number of issues (see e.g{.}~\cite{goldberg}). Foremost of these is the fact that the (floating-point) bound thus obtained is neither exact nor entirely rigorous. \flagmatic\ offers two ways around this difficulty.

If the floating-point bound is not thought to be tight, then the simplest of the two ways is also the most appropriate: \flagmatic\ can perform a Cholesky decomposition of the matrices, and then round off each entry to the nearest rational, with denominators bounded by a suitable integer $q$ ($q={10}^8$ is the default, if the user does not supply a preference). In this way, a rational bound on the Tur\'an density $\pi(\mathcal{F})$ can be obtained rigorously. The said bound may appear to be slightly worse than the floating-point bound initially reported by \flagmatic, but in practice we may keep this discrepancy below ${10}^{-6}$ by choosing $q$ large enough.

On the other hand, if the floating-point bound first reported by \flagmatic\ is thought to be tight, and if we know a matching lower bound construction, then we can do better. Given a lower bound construction, \flagmatic\ will use it to construct zero eigenvectors of the positive semi-definite matrices found by the SDP solver. This is done by using Remark~\ref{zeroevec}. 
So for each positive semi-definite matrix $Q$, assuming that all the zero eigenvectors can be obtained in this way,  we can factor out the zero eigenspace and write $Q$ as a product 
\[ Q = R \ Q' \ R^T \]
where $Q'$ is positive definite. Moreover, because the $R$ matrix can be constructed by considering, loosely speaking, `flag densities in the limit of an extremal configuration,' it can be constructed with rational entries. \flagmatic\ then rounds the entries of $Q'$ to nearby rationals, its choices being guided in a few cases by the conjectured value of $\pi(\mathcal{F})$. (The rounding procedure used by \flagmatic\ is somewhat unsophisticated, but we have found it to be sufficient for our purposes. More complicated methods of rounding are possible, for example one could try to minimise the Euclidean distance between the original floating-point matrix and the rounded matrix, as proposed in Section 2.4.2 of \cite{BaberThesis}.)

Since the floating-point matrix $Q'$ is positive definite, the `rounded off' matrix will also be positive definite, provided our approximation is sufficiently fine. (Indeed if the perturbation of the entries of $Q'$ introduced in the rounding-off process is too great, \flagmatic\ will report an error and ask to use larger denominators $q$.) Finally, to ensure that it is beyond doubt that the `rounded off' $Q'$ is positive definite, \flagmatic\ uses a change of basis (via Gaussian elimination) to put it in diagonal form. (The $R$ matrix is modified so that $Q=R \ Q' \ R^T$ is unchanged.)

Finally, \flagmatic\ will produce a certificate of the rigorous flag algebra bound~(\ref{fab}), of which more will be said in the next subsection. For more information about using \flagmatic, see \cite{usersguide}.

\subsection{Certificates}

One of the drawbacks of the flag algebra method is that computations rapidly become very involved. The number of distinct 3-graphs on $n$ vertices, up to isomorphism, for $n=1,2,\dots$ grows very rapidly:
\[1, \ 1, \ 2, \ 5, \ 34, \ 2136, \ 7013320, \ \dots\]
(sequence A000665 of \cite{oeis}), and the size of the family of admissible graphs increases at a comparable pace in most problems. In practical terms, this means that we cannot perform any flag algebra calculations with admissible graphs of order $m>7$, and that even for $m=6$ and $m=7$, many flag algebra calculations involve too many graphs to be easily verifiable by hand.

Different authors have used different ways of addressing this issue: some~\cite{grzesik, hirst, Razborov2} include lists of admissible graphs, intersection types, flags and large positive semi-definite matrices in the body of their papers; others~\cite{BaberTalbot} worked with matrices that were too large and admissible graphs that were too numerous for this to be a practical solution, and omitted them from their papers. Our calculations by and large fall in the latter category, and we will similarly omit long lists of data.

Instead, we have used \flagmatic\ to produce certificates for all the flag algebra calculations we perform. These certificates are available on our website
\begin{quote}
\website
\end{quote}
as well as in the ancillary files associated with the arXiv version of this paper.
The certificates are in the JSON format \cite{json}, which is designed to be human-readable. Let us give details of what they contain, and of how this may be used to verify our calculations.

\flagmatic\ uses the following notation for 3-graphs. First the order $n$ is given, followed by a colon and a (possibly empty) list of 3-edges, which are given as a string of numbers $x_1 y_1 z_1 x_2 y_2 z_2 \dots$. For example, ``\verb|3:|'' represents the 3-graph on 3-vertices with no edges, whilst ``\verb|4:123124134|'' and ``\verb|4:213214234|'' both represent $K_4^-$.

All numbers in the certificates are rational, and are either provided as fractions ``p/q'', or as integers. Symmetric matrices are given by the entries in their upper triangle, so that

\begin{center} \verb|[[1,0,0], [1,0], [1]]|
\end{center}

is the $3 \times 3$ identity matrix. Matrices that are not necessarily symmetric are given by their rows, with 

\begin{center} \verb|[[1,-2],[-5,3]]|
\end{center}

standing for the matrix
\[\begin{pmatrix} 1 & -2 \\ -5 & 3\\ \end{pmatrix}.\]

The certificates produced by \flagmatic\ contain the following information:
\begin{enumerate}
\item A description of the problem, specifying which $r$-graphs we are working with (in all our applications, $r=3$); what we are trying to maximise (in this paper, the density of 3-edges, referred to as ``\verb|3:123|'' in the certificate); and which configurations we are forbidding.
\item The bound obtained (a rational number).
\item The order $m$ of the admissible graphs we are working with; the number of admissible graphs of order $m$ (up to isomorphism); and a list of the admissible graphs in the \flagmatic\ notation.
\item The number of intersection types used; and a list of the intersection types in the \flagmatic\ notation.
\item A list of the number of flags for each intersection type (the first number in the list corresponding to the first intersection type listed, the second number corresponding to the second intersection type, and so on); and a list of the $\sigma$-flags for each type $\sigma$ (in \flagmatic\ notation, ordered by type as above).
\item A list of $Q'$ matrices (called ``\verb|qdash_matrices|'' in the certificate), one for each intersection type.
\item A list of $R$ matrices (called ``\verb|r_matrices|'' in the certificate), one for each intersection type.
\end{enumerate}

At this stage the reader way wonder why we are giving two matrices for each intersection type, rather than just one. Recall that for each intersection type $\sigma$ we must provide a positive semi-definite matrix $Q$ to use in inequality (\ref{singletype}). To ensure that there can be no doubt as to the positive semi-definiteness of the matrices it provides, \flagmatic\ gives two matrices $R$ and $Q'$ where $Q'$ is a positive definite \emph{diagonal} matrix and $R$ is a rectangular matrix. The matrix $Q$ is then computed as
\[ Q = R \, Q' \, R^T. \]

Given all this information, what does one need to do to verify that the flag algebra calculation is indeed correct? There are four stages:
\begin{enumerate}
\item First of all, one needs to check that the family of admissible 3-graphs given in the certificate is indeed the family of \emph{all} admissible 3-graphs of order $m$.
\item For all admissible graphs $H$ and all intersection types $\sigma$, one then needs to compute the densities $d(H)$  and the flag pair densities $d_{F,F'}(H)$ for all each pair of $\sigma$-flags $(F,F')$.
\item Next, the $Q$ matrices must be computed from the $Q'$ and $R$ matrices.
\item Finally, one needs to substitute all these terms into inequality~(\ref{fab}) and check that the claimed bound is achieved.
\end{enumerate}

To assist with these tasks, we provide a separate checker program, available from the \flagmatic\ website, called ``inspect\_certificate.py''. This program is independent of \flagmatic, and only requires Python 2.6 or 2.7 to run. Given a certificate as input, it can do any of the following:

\begin{itemize}
\item Display the list of admissible graphs.
\item Display the types and flags.
\item Display the $Q'$ and $R$ matrices.
\item Compute and display the $Q$ matrices.
\item Compute and display the admissible graph densities.
\item Compute and display the flag pair densities.
\item Compute and display the flag algebra coefficients for each admissible graph.
\item Compute and display which admissible graphs have a flag algebra coefficient equal to the bound.
\end{itemize}

As mentioned earlier, the certificates for our results are available on the \flagmatic\ website, and in a data set included in our arXiv submission. 
Each certificate has a unique file name, which is given in the following table:

\begin{center} {\small
\begin{tabular}{ll}
Result & Certificate filename(s) \\
\hline
Theorem \ref{k4-c5f32} & \verb|k4-f32c5.js| \\
Theorem \ref{k4-f32} & \verb|k4-f32.js| \\
Theorem \ref{j4f32} & \verb|38.js| \\
Theorem \ref{j5f32} & \verb|638.js| \\
Theorem \ref{oddcyclesbounds} & \verb|k4-l5.js| and \verb|k4-f32l5.js| \\
Theorem \ref{weak sos conj} & \verb|k58i.js| \\
Theorem \ref{inducedk4-f32} & \verb|43if32.js| \\
Theorem \ref{k4 d3 bound} & \verb|k4j4.js| \\
Proposition \ref{k4-c5} \ & \verb|k4-c5.js| \\
Proposition \ref{itblowupbounds} \ & \verb|k4-.js|, \verb|c5.js| and \verb|blm.js| \\
\hline
\end{tabular} }
\end{center}

\section{Results}

\subsection{On the extremal theory of 3-graphs with independent neighbourhoods}

A 3-graph $G$ is said to have \emph{independent neighbourhoods} if for every $x,y \in V(G)$ the \emph{joint neighbourhood} \[\Gamma(x,y)=\{z: \ xyz \in E(G)\}\] of $x$ and $y$ is an edge-free set in $G$. This is equivalent to saying that $G$ contains no copy of $F_{3,2}$ as a subgraph, where $F_{3,2}=([5], \{123, 124, 125,345\})$. For reasons we shall elaborate on in Section 4.1, the extremal theory of 3-graphs with independent neighbourhoods is very amenable to flag algebra calculus-based investigations.

The first result we should mention is due to F\"uredi, Pikhurko and Simonovits~\cite{FurediPikhurkoSimonovits}, who established the Tur\'an density of $F_{3,2}$ (and in fact determined its Tur\'an number $\ex(n, F_{3,2})$ exactly).

\begin{theorem}[F\"uredi, Pikhurko, Simonovits]
\[\pi(F_{3,2})=4/9.\] \label{fps} \end{theorem}

The next four results are new however.

\begin{theorem}\label{k4-c5f32}
\[\pi(K_4^-, C_5, F_{3,2})=12/49.\]
\end{theorem}
\begin{proof}

The upper bound is from a flag algebra calculation using \flagmatic\ (see Section 2.4 for how to obtain a certificate). The lower bound, which was independently obtained by F\"uredi~\cite{Mubayiemail}, comes from taking a balanced blow-up of the 6-regular 3-graph on $7$ vertices
\[H_7=([7], \{124,137,156,235,267,346,457,653,647,621,542,517,431,327\}).\]
The 3-graph $H_7$ can be obtained as the union of two edge-disjoint copies of the Fano plane on the same vertex set
\begin{align*}
F_1&=([7], \{124,137,156,235,267,346,457\}) \textrm{ and }\\
F_2 &=([7], \{653,647,621,542,517,431,327\}),
\end{align*}
as depicted in Figure \ref{furedifano}. This elegant perspective is due to F\"uredi~\cite{Mubayiemail}.

Another way to think about $H_7$ is by considering its link-graphs: for every $i \in [7]$, the link graph of $i$ in $H_7$ is a $6$-cycle, which is triangle-free (in fact bipartite). This instantly shows that a blow-up of $H_7$ is $K_4^-$-free. To see that $H_7$ and its blow-ups are $F_{3,2}$ free, it is enough to observe that for every $i\neq j$ in $[7]$, the codegree of $i$ and $j$ in $H_7$ is exactly $2$, which is not enough to support a 3-edge, so that their joint neighbourhood remains edge-free in the blow-up. Finally, to see that such a blow-up is $C_5$-free, note that $H_7$ is itself $C_5$-free, so that 5 vertices in distinct parts of the blow-up cannot span a $C_5$, while, on the other hand, a copy of $C_5$ in the blow-up cannot involve two vertices in the same part (since any two vertices of $C_5$ appear together in a 3-edge). 
\end{proof}

\begin{figure}
\begin{center}
\begin{tikzpicture}
\tikzstyle{blob} = [circle, draw, fill=white, text centered, inner sep=1pt]	
\draw[black] (3/2,\sinsixty) circle (\sinsixty);
\draw[black,fill=black] (0,0) circle (2pt) node[blob] {$5$} --
(3,0) circle (2pt) node[blob] {$6$} --
(3/2,3*\sinsixty) circle(2pt) node[blob] {$3$} -- (0,0)
(0,0) -- (30:\midfano) circle (2pt) node[blob] {$4$} (3,0) --
+(150:\midfano) circle (2pt) node[blob] {$2$}
(3/2,3*\sinsixty)--(3/2,0) circle (2pt) node[blob] {$1$}
(3/2,\sinsixty) circle (2pt) node[blob] {$7$};
\begin{scope}[xshift=4cm]
\draw[black] (3/2,\sinsixty) circle (\sinsixty);
\draw[black,fill=black] (0,0) circle (2pt) node[blob] {$2$} --
(3,0) circle (2pt) node[blob] {$1$} --
(3/2,3*\sinsixty) circle(2pt) node[blob] {$4$} -- (0,0)
(0,0) -- (30:\midfano) circle (2pt) node[blob] {$3$} (3,0) --
+(150:\midfano) circle (2pt) node[blob] {$5$}
(3/2,3*\sinsixty)--(3/2,0) circle (2pt) node[blob] {$6$}
(3/2,\sinsixty) circle (2pt) node[blob] {$7$};
\end{scope}
\end{tikzpicture}
\end{center}
\caption{F\"uredi's double Fano construction.} \label{furedifano}
\end{figure}
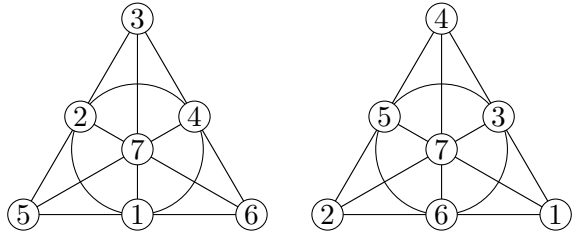

The next result is similar to an earlier theorem of Frankl and F\"uredi~\cite{FranklFuredi}, which we shall discuss in the next section, where we also show how our results differ.
\begin{theorem}\label{k4-f32}
\[\pi(K_4^-, F_{3,2})=5/18.\]
\end{theorem}
\begin{proof}
The upper bound is from a flag algebra calculation using \flagmatic\ (see Section 2.4 for how to obtain a certificate). The lower bound, due to Frankl and F\"uredi, is obtained by taking a balanced blowup of the following 5-regular 3-graph on $6$ vertices,
\[H_6=([6], \{123, 234, 345, 145, 125, 136, 356, 256, 246, 146\}). \]
There are two easy ways to visualise $H_6$. On the one hand, it is the unique 3-graph on 6 vertices such that for every $i \in [6]$ the link graph of $i$  is a $5$-cycle. Alternatively, we may think of it as the unique 3-graph on 6 vertices with all its 5-vertex subgraphs isomorphic to $C_5$. The first description makes it clear that blow-ups of $H_6$ are $K_4^-$-free, since the link graphs of $H_6$ contain no triangles. A blow-up of $C_5$ clearly has independent neighbourhoods, and a copy of $F_{3,2}$ involves vertices in at most $5$ different parts of a blow-up, so the second description establishes that blow-ups of $H_6$ are $C_5$-free as well.
\end{proof}

In both Theorems~\ref{k4-c5f32} and~\ref{k4-f32}, we believe that the lower bound construction given is the stable extremal configuration. 

\begin{theorem}\label{j4f32}
\[\pi(J_4, F_{3,2})=3/8.\]
\end{theorem}
\begin{proof}The upper bound is from a flag algebra calculation using \flagmatic\ (see Section 2.4 for how to obtain a certificate). The lower bound is obtained by taking a balanced blow-up $H$ of $K_4$. For each vertex $x$ in the resulting 3-graph, the link graph is the disjoint union of an independent set of vertices and a complete 3-partite graph; such a graph clearly cannot contain a complete graph on 4 vertices, establishing that $H$ is $J_4$-free. To see that $H$ is $F_{3,2}$-free as well, it is enough to note that a copy of $F_{3,2}$ cannot involve two vertices lying in the same part of $H$, and that $H$ has only $4$ parts whereas  $F_{3,2}$ has $5$ vertices.
\end{proof}

\begin{theorem}\label{j5f32}
\[\pi(J_5, F_{3,2})=3/8.\]
\end{theorem}
\begin{proof}
The upper bound is from a flag algebra calculation using \flagmatic\ (see Section 2.4 for how to obtain a certificate). The lower bound is obtained, as in Theorem~\ref{j4f32}, by taking a balanced blow-up $H$ of $K_4$. Since $J_4$ is a subgraph of $J_5$, and $H$ is $J_4$-free, $H$ must be $J_5$-free as well. 
\end{proof}

We should make two remarks here. First of all, the flag algebra calculation involved in the proof of Theorem~\ref{j4f32} is `easy' in comparison with the calculations involved in the proofs of Theorems~\ref{k4-c5f32} and~\ref{k4-f32}. This, and the pleasing structure of our lower bound construction, suggest that the underlying Tur\'an density problem should be amenable to more direct combinatorial arguments. Secondly, we might have expected that the extremal configuration for the $(J_5, F_{3,2})$ problem be a balanced blow-up of $K_5$, yielding link graphs consisting of complete 4-partite graphs together with an independent set. However, $K_5$ is not $F_{3,2}$-free, and as Theorem~\ref{j5f32} shows, we do not gain anything from forbidding $J_5$ rather than $J_4$. It seems natural to ask whether this changes if one forbids $J_t$, for some $t>5$.

\begin{question}
Is it the case that for all $t \geq 4$,
\[\pi(J_t, F_{3,2})=3/8?\]
\end{question}

Let us also remark that all previous known results in extremal 3-graph theory had one of five extremal configurations: the blow-up of a 3-edge~\cite{bela, franklfuredi}, $H_6$~\cite{FranklFuredi}, the `one-way' complete bipartite 3-graph~\cite{FurediPikhurkoSimonovits} (an unbalanced blow-up of the degenerate 3-graph $([2],\{112\})$), Tur\'an's construction~\cite{Razborov2} (where the proof also relied on the flag algebra calculus) and the complete bipartite 3-graph~\cite{BaberThesis, deCaenFuredi, MubayiRodl, KeevashMubayi}. We can now add two more extremal configurations to this list: the balanced blow-up of $H_7$ and the balanced blow-up of $K_4$.

We now come to some Tur\'an problems for which we have been unable to find tight bounds using \flagmatic. Erd\H{o}s and S\'os conjectured that the maximal density of a 3-graph in which all vertices have a bipartite link graph is $1/4$:
\begin{conjecture}[Erd\H os, S\'os~\cite{ErdosSos}]\label{ErdosSos}
\[\pi(\Text{odd cycle in link graph})=1/4.\]
\end{conjecture}
If the conjecture is true, then this is an extremely unstable problem. Two different constructions were given by Frankl and F\"uredi~\cite{FranklFuredi}:
\begin{construction}[Frankl, F\"uredi]\label{c1}
Distribute $n$ vertices uniformly along the circumference of a circle. Then define a 3-graph on $n$ vertices by putting a 3-edge $xyz$ in the graph if the centre of the circle lies in the interior of the triangle determined by $x$, $y$ and $z$, to obtain a $K_4^-$-free 3-graph.
\end{construction}
\begin{construction}[Frankl, F\"uredi]\label{c2}
Consider a random tournament $T$ on $n$ vertices. Then define a 3-graph on $n$ vertices by putting a 3-edge $xyz$ in the graph if $xyz$ is an oriented triangle in $T$.
\end{construction}
To these constructions, we can add five more
\begin{construction}\label{c3}
Take a balanced, iterated blow-up of the 3-graph consisting of a single 3-edge, $G=([3], \{123\})$.
\end{construction}
\begin{construction}\label{c4}
Take a balanced iterated blow-up of $C_5$.
\end{construction}
\begin{construction}\label{c5}
Take a balanced iterated blow-up of $H_7$.
\end{construction}
\begin{construction}\label{c6}
Take a balanced iterated blow-up of 
\[([7], \{123, 124,125,136,137,146,247,256,257,347,356,357,456,467\}).\]
\end{construction}
\begin{construction}\label{c7}
Take a balanced iterated blow-up of 
\[([7], \{123, 124,125,136,146,157,237,247,256,345,356,367,457,467\}).\]
\end{construction}
The last three constructions are all iterated blow-ups of some 6-regular 3-graph on 7 vertices. The best way to think about them is perhaps in terms of their link graphs: the link graphs in $H_7$ consist of 6-cycles, whereas the links in Constructions \ref{c6} and \ref{c7} are isomorphic to (respectively) a 4-cycle with two pendant edges attached to a pair of adjacent vertices, and a 4-cycle with a path of length 2 attached to one of the vertices. In all three cases, the links are bipartite, and so the links in an iterated blow-up are bipartite as well.

In fact, more generally, if $G$ is a 3-graph with bipartite links, then any iterated blow-up of $G$ also has bipartite links. We can thus construct arbitrarily many non-isomorphic configurations of 3-graphs with bipartite links and 3-edge density $1/4 +o(1)$ by taking any of the above constructions, blowing it up, and then inside each of the parts, we are free to place a copy of any of the other constructions.

Given this instability, the bipartite links conjecture of Erd\H{o}s and S\'os appears very hard. We believe, however, that the independent neighbourhoods version of the problem should be stable with Construction~\ref{c1} being the essentially unique extremal configuration.

\begin{conjecture}\label{indep odd cycle}
\[\pi(\Text{odd cycle in link graph}, F_{3,2})=1/4,\]
with the stable extremal configuration being given by Construction~\ref{c1}. 
\end{conjecture}

\begin{figure}
\begin{center}
\usetikzlibrary{arrows}
\begin{tikzpicture}
\tikzstyle{vertex} = [circle, draw, fill=white, text centered, inner sep=1pt]
\foreach \a in {1, ..., 5} \draw (72*\a+90:1.8) node[vertex] (n\a) {$x_\a$};
\draw[-latex] (n2)--(n1);
\draw[-latex] (n1)--(n3);
\draw[-latex] (n3)--(n2);
\draw[-latex] (n1)--(n4);
\draw[-latex] (n4)--(n2);
\draw[-latex] (n1)--(n5);
\draw[-latex] (n5)--(n2);
\draw[-latex] (n3)--(n4);
\draw[-latex] (n4)--(n5);
\draw[-latex] (n5)--(n3);
\end{tikzpicture}
\end{center}
\caption{An orientation of $K_5^{(2)}$.} \label{k5orientation}
\end{figure}
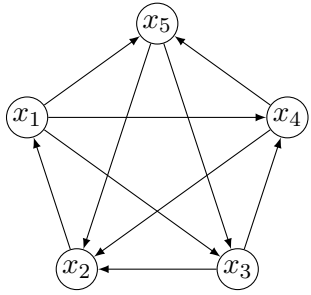

In fact, more generally, we believe that extremal problems for 3-graphs with independent neighbourhoods should be stable:

\begin{conjecture}
Tur\'an problems for 3-graphs with independent neighbourhoods are stable.
\end{conjecture}

As we shall see in Section 3.3 however, the extremal theory of 3-graphs with independent neighbourhoods still has non-principality: there exist 3-graphs $H_1$ and $H_2$ such that 
\[\pi(H_1, H_2, F_{3,2})< \min\left(\pi(H_1, F_{3,2}), \ \pi(H_2, F_{3,2})\right).\]
Thus even in this restricted setting we cannot hope for an analogue of the Erd\H os-Stone Theorem from extremal graph theory.

Before we close this section, let us note the bounds we can obtain using \flagmatic\ for the problems in Conjectures~\ref{ErdosSos} and~\ref{indep odd cycle}:

\begin{theorem} \label{oddcyclesbounds}
\[1/4 \leq\pi(\Text{odd cycle in link graph}, F_{3,2}) < 0.255889, \text{ and}\]
\[1/4 \leq \pi(\Text{odd cycle in link graph}) < 0.258295.\]
\end{theorem}
\begin{proof}
The upper bounds are from two flag algebra calculations using \flagmatic\ (see Section 2.4 for how to obtain a certificate). Lower bounds from Construction~\ref{c1}.
\end{proof}

Let us finally outline a proof of our claim that Constructions~\ref{c1}--\ref{c7} are distinct. (That they have asymptotic density $1/4$ and bipartite links is left as an exercise for the reader.)

Constructions~\ref{c3}--\ref{c7} can be distinguished by considering their link-graphs; they are moreover highly structured, so that with high probability, the random Construction~\ref{c2} cannot be edited into them without changing at least a constant proportion of the 3-edges. (Indeed the probability of say $n/3$ vertices having identical neighbourhoods (up to $o(n^3)$ edges) in the rest of the 3-graph is exceeding small.)

Clearly iterated blow-up constructions are not $F_{3,2}$-free. It is easy to see that Construction~\ref{c2} is not $F_{3,2}$-free either: given 5 vertices $x_1,x_2,x_3,x_4,x_5$, the orientation 
\[\overrightarrow{x_2x_1},\overrightarrow{x_1x_3}, \overrightarrow{x_3x_2},\overrightarrow{x_1x_4}, \overrightarrow{x_4x_2}, \overrightarrow{x_1x_5},\overrightarrow{x_5x_2},\overrightarrow{x_3x_4}, \overrightarrow{x_4x_5}, \overrightarrow{x_5x_3}\]
(see Figure \ref{k5orientation}) occurs with probability at least $2^{-10}$ in a random tournament, so that we expect $F_{3,2}$ to occur as a subgraph in Construction~\ref{c2} with strictly positive density.

Now on the other hand, Construction~\ref{c1} \emph{is} $F_{3,2}$-free. Indeed, consider any two vertices $x_1$ and $x_2$ on the circumference of a circle, and let us show that their common neighbourhood is an independent set. If $x_1$ and $x_2$ lie on the same diameter, their codegree must be zero, as $x_1$ and $x_2$ cannot be vertices of a triangle that has the centre of the circle in its interior. Thus we may assume, without loss of generality, that $x_1$ and $x_2$ do do not lie on the same diameter. Then the diameters through $x_1$ and $x_2$ separate the circumference of the circle into four arcs (see Figure \ref{circleconst}). By construction, the common neighbourhood of $x_1$ and $x_2$ consists of all the vertices lying on the interior of the arc that contains neither $x_1$ nor $x_2$. But by construction this is an independent set of vertices. Thus Construction~\ref{c1} is distinct from all our other constructions.

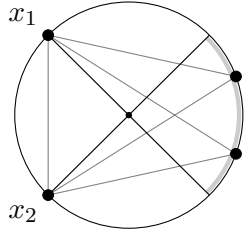
\begin{figure}
\begin{center}
\begin{tikzpicture}
\draw[ultra thick, gray!40] (45:14.5mm) arc (45:-45:14.5mm);
\draw[gray] (135:15mm)--(20:15mm)--(225:15mm)--(135:15mm)--(-20:15mm)--(225:15mm);
\draw (0,0) circle (15mm);
\draw[fill=black] (0,0) circle (1pt);
\draw[fill=black] (0,0) -- (135:15mm) circle (2pt) node[above left] {$x_1$}
(0,0)--(-45:15mm) (0,0)--(45:15mm) (0,0)--(225:15mm)
circle (2pt) node[below left] {$x_2$} (20:15mm) circle (2pt) (-20:15mm) circle (2pt);
\end{tikzpicture}
\end{center}
\caption{The circle construction has independent neighbourhoods.} \label{circleconst}
\end{figure}

\subsection{Forbidding induced subgraphs}
Tur\'an's conjecture is arguably the most famous open problem in extremal combinatorics.

\begin{conjecture}[Tur\'an]\label{turanconj}
\[\pi(K_4)=5/9.\]
\end{conjecture}

Tur\'an's original construction for the lower bound that motivates his conjecture is obtained by taking a balanced tripartition $A\sqcup B \sqcup C$ of the vertex set, and putting in all 3-edges of type $AAB, BBC,CCA$ and $ABC$. (In our language, this is a blow-up of the degenerate 3-graph $([3], \{123, 112, 223, 331\})$.) Many other other constructions for the problem have since been found. Indeed there are exponentially many nonisomorphic 3-graph configurations on $n$ vertices attaining the bound given by Tur\'an's construction while not containing any copy of $K_4$: see Brown~\cite{Brown}, Kostochka~\cite{Kostochka}, Fon-der-Flaas~\cite{FonderFlaass} and Frohmader~\cite{Frohmader}. 
If Tur\'an's conjecture is true, the Tur\'an density problem for $K_4$ is therefore very unstable and thus (for reasons we shall develop in Section~4) unlikely to be resolved by a pure flag algebra calculus-based approach.

Razborov observed, however, that Tur\'an's construction is the only one in which no 4-set of vertices spans exactly one 3-edge. Adding this restriction, he was able to give a proof of a weakening of Tur\'an's conjecture using the flag algebra calculus. Formally, let us call $G_1$ the unique (up to isomorphism) 3-graph on 4 vertices with exactly one 3-edge. Then the following holds:

\begin{theorem}[Razborov~\cite{Razborov2}]
\[\pi(K_4, \Text{ induced }G_1)=5/9.\]
\label{razb41}
\end{theorem}

Thus in this case Razborov was able to circumvent the instability of the $K_4$ problem to obtain his result.
Further work on the Tur\'an conjecture along these lines can be found in~\cite{Pikhurko, Razborov3}. Proceeding similarly to Razborov, we considered the following conjecture, which is also attributed to Tur\'an:

\begin{conjecture}[Tur\'an]\label{k5conj}
\[\pi(K_5)=3/4.\]
\end{conjecture}

As in Conjecture~\ref{turanconj}, more than one extremal configuration attaining the conjectured bound is known. One $K_5$-free 3-graph with density $3/4+o(1)$ is obtained by taking a complete balanced bipartite 3-graph. Another example, due to Keevash and Mubayi~\cite{Keevash}, is obtained by taking a balanced blow-up of $K_4$ and, writing $A \sqcup B \sqcup C \sqcup D$ for the corresponding 4-partition of the vertex sets, adding all 3-edges of type AAB, BBC, CCD and DDA as well as all 3-edges of type AAC, CCA, BBD and DDB. In our notation, this is a blow-up of the degenerate 3-graph \[([4], \{123,124, 134, 234, 112, 223, 334, 441, 113, 331, 224, 442\}.\]

This is easily seen to be distinct from the first example. Many more configurations exist: Sidorenko exhibited infinite families of nonisomorphic $K_5$-free constructions with asymptotic density $3/4$ (see Constructions 4--7 in~\cite{Sidorenko}). Thus if Conjecture~\ref{k5conj} is true, then the Tur\'an density problem for $K_5$ is very unstable and, just as in Conjecture~\ref{turanconj}, we are unlikely to arrive at tight bounds for $\pi(K_5)$ by a pure flag algebra calculus-based approach.

We are, however, able to obtain an analogue of Razborov's result: observe that in a complete bipartite graph, a 5-set of vertices cannot span exactly 8 edges. On the other hand, consider for example the construction of Keevash and Mubayi: taking two vertices from part A, one from part B and two from part C yields a 5-set spanning exactly 8 edges. Let us therefore write $\mathcal{G}$ for the family of 3-graphs on 5 vertices with exactly 8 edges. (There are only two such 3-graphs up to isomorphism; considering $K_5$ as a graph on the vertex set $[5]$, these are $K_5\setminus\{123,145\}$ and $K_5 \setminus \{123,124\}$ respectively.) Then the following holds:

\begin{theorem}\label{weak sos conj}
\[\pi(K_5, \Text{ induced copy of a member of $\mathcal{G}$})=3/4.\]
\end{theorem}
\begin{proof}
The upper bound is from a flag algebra calculation using \flagmatic\ (see Section 2.4 for how to obtain a certificate). The lower bound is from consideration of a complete balanced bipartite 3-graph.
\end{proof}

Just as for Theorem \ref{razb41}, it would be nice to have a more direct, combinatorial proof of Theorem~\ref{weak sos conj}; the proof above does not give much insight into the problem.

The strategy of introducing extra restrictions that we know must be satisfied by our desired extremal configuration in order to obtain a better bound is not new. An earlier result of a similar flavour (but proved without resorting to the flag algebra calculus) is the following Theorem of Frankl and F\"uredi~\cite{FranklFuredi}.
\begin{theorem}[Frankl, F\"uredi~\cite{FranklFuredi}]\label{ff}
\[\pi(K_4^-, \Text{ induced }G_1)= 5/18.\]
\end{theorem}
In fact Frankl and F\"uredi showed rather more: they determined the Tur\'an number for this problem and showed the unique extremal graph is a balanced blow-up of $H_6$. Even more, they proved that all $K_4^-$-free 3-graphs with no induced copy of $G_1$ are either (possibly unbalanced) blow-ups of $H_6$ or are of the form given by Construction~\ref{c1} in the previous subsection.

An attentive reader will observe that the density version of Frankl and F\"uredi's result which we stated above is very similar to Theorem~\ref{k4-f32}. Indeed, the two results share the same lower bound construction. Let us observe that forbidding a 3-graph from containing a copy of $K_4^-$ or $F_{3,2}$ is strictly weaker than forbidding a 3-graph from containing a copy of $K_4^-$ or an induced copy of $G_1$ (which is equivalent to requiring that all 4-sets span exactly 0 or 2 edges). Theorem~\ref{k4-f32} is thus a nominally stronger result than Theorem~\ref{ff}.

\begin{lemma}
Suppose $G$ is a 3-graph in which 4-sets span exactly 0 or 2 edges. Then $G$ is $(K_4^-,F_{3,2})$-free. The converse is false.
\end{lemma}
\begin{proof} Let $G$ be a 3-graph in which 4-sets span exactly 0 or 2 edges. Then $G$ is trivially $K_4^-$-free. Suppose it contained $F_{3,2}$ as a subgraph. By relabelling vertices, we have that $G$ contains $5$ vertices $a,b,c,d,e$ such that $abc, ade, bde, cde$ are all edges of $G$. Now the 4-set $\{a,b,d,e\}$ already spans 2 edges, so it cannot span any more. Thus neither of $abd,abe$ lies in $E(G)$. Similarly, none of $acd,ace$ and $bcd,bce$ can lie in $E(G)$. Now consider the 4-set $abcd$. This spans exactly one edge, the other three having been forbidden; but this contradicts the fact that $G$ is a 3-graph in which 4-sets  span exactly 0 or 2 edges.

To see that the converse is false, consider a 3-graph on 4 vertices with 1 edge. This is obviously $(K_4^-, F_{3,2})$-free but violates the condition that $4$-sets span exactly 0 or 2 edges. The same is true of any of its blow-ups.
\end{proof}

Finally, let us stress just how different forbidding induced subgraphs is to forbidding subgraphs. We have shown that $\pi(K_4^-, F_{3,2})=5/18$. In marked contrast is the following:
\begin{theorem} \label{inducedk4-f32}
\[\pi(\Text{induced }K_4^-, \ F_{3,2})=3/8.\]
\end{theorem}
\begin{proof}
The upper bound is from a flag algebra calculation using \flagmatic\ (see Section 2.4 for how to obtain a certificate). The lower bound is from consideration of a balanced blow-up of $K_4$.
\end{proof}

Note Theorems~\ref{inducedk4-f32} and~\ref{j4f32} are implied by Theorem~\ref{j5f32} and the observation that a blow-up of $K_4$ is $F_{3,2}$-free, $J_4$-free and contains no induced $K_4^-$. Indeed, suppose an $F_{3,2}$-free $3$-graph $G$ contains a copy of $J_5$. This consists of a 5-set $S$ together with a vertex $x \notin S$ and all $\size{S^{(2)}}$ possible $3$-edges containing $x$ and two vertices from $S$. Since $G$ is $F_{3,2}$-free, it must also be $K_5$-free, and hence at least one $3$-edge $e=\{abc\}$ from $S^{(3)}$ is missing in $G$. The $4$-set $\{xabc\}$ then spans an induced copy of $K_4^-$ in $G$.

\subsection{Nonprincipal pairs}
By Theorem~\ref{k4-f32}, $\pi(K_4^-, F_{3,2})=5/18$. On the other hand, Frankl and F\"uredi gave a lower bound of $2/7$ for $\pi(K_4^-)$ by considering a balanced iterated blow-up of $H_6$~\cite{FranklFuredi},
while F\"uredi, Pikurkho and Simonovits~\cite{FurediPikhurkoSimonovits} showed
$\pi(F_{3,2})=4/9$. Gathering all this together we have:
\[\pi(K_4^-,F_{3,2})=\frac{5}{18}< \min\left(\pi(K_4^-),\pi(F_{3,2})\right).\]
Thus $(K_4^-, F_{3,2})$ is an example of a \emph{nonprincipal} pair of 3-graphs---that is to say, a pair $F,F'$ with $\pi(F,F') < \min(\pi(F), \pi(F'))$. Nonprincipality for 3-graphs was conjectured by Mubayi and R\"odl~\cite{MubayiRodl} and first exhibited by Balogh~\cite{Balogh}. Mubayi and Pikhurko~\cite{MubayiPikhurko} then built on Balogh's ideas to give the first example of a nonprincipal pair of 3-graphs, and Razborov~\cite{Razborov2} used his flag algebra method to show $(K_4^-,C_5)$ is also a nonprincipal pair. We can exhibit yet another nonprincipal pair of 3-graphs:

\begin{theorem}\label{k4 d3 bound}
\[\pi(K_4, J_4) < 0.479371 < 1/2 \leq \pi(J_4).\]
\end{theorem}
\begin{proof}
The upper bound on $\pi(K_4, J_4)$ is from a flag algebra calculation using \flagmatic\ (see Section 2.4 for how to obtain a certificate). The lower bound for $\pi(J_4)$, due to Bollob\'as, Leader and Malvenuto~\cite{BollobasLeaderMalvenuto}, is a balanced iterated blow-up of the complement of the Fano plane.
\end{proof}

Given that $\pi(K_4)\geq 5/9$, it follows that $(K_4, J_4)$ is a fourth nonprincipal pair of 3-graphs. (It is in fact very similar to the example given by Mubayi and Pikhurko~\cite{MubayiPikhurko}, who showed that $(K_4, J_5)$ is a nonprincipal pair.) Note that we can show $\pi(K_4, J_4) \ge 2/5$ by considering an iterated blow-up of 
\[([6], \{123,124,125,134,135,146,156,236,245,246,256,345,346,356\}),\]
but $2/5$ is quite far from the upper bound.

\begin{question} \label{k4j4question}
What is $\pi(K_4, J_4)$?
\end{question}

Finally, let us remark that the extremal theory of 3-graphs with independent neighbourhoods also exhibits nonprincipality: by Theorems~\ref{k4-c5f32} and~\ref{k4-f32},
\[\pi(K_4^-,C_5, F_{3,2})=12/49< \pi(K_4^-,F_{3,2})=5/18<\pi(C_5, F_{3,2})=4/9,\]
where in the last line we have used the fact that $\pi(C_5,F_{3,2})=\pi(F_{3,2})$ (which holds since the extremal configuration for $F_{3,2}$ is $C_5$-free.) Thus even in the case of 3-graphs with independent neighbourhoods we cannot hope for some analogue of the Erd\H os-Stone theorem from extremal graph theory.

Nonprincipality is in general hard to prove by hand; it can however be a useful tool to know when attacking Tur\'an density problems: a common strategy when studying $\pi(\mathcal{F})$ for some family $\mathcal{F}$ is to try showing that $\pi(\mathcal{F},G)$ is less than the conjectured valued of  $\pi(\mathcal{F})$ for some nice, dense 3-graph $G$, and then use the presence of a (large) number of copies of $G$ in a putative $\mathcal{F}$-extremal example to bound the edge-density. (See for example~\cite{deCaenFuredi} for a nice example of this technique.) So provided that $\pi(\mathcal{F},G)< \pi(\mathcal{F}) \leq \pi(G)$ is actually true, that we have a (conjectured) extremal $\mathcal{F}$-free construction, and that $G$ and the graphs in $\mathcal{F}$ are not too large, \flagmatic\ can be expected to show nonprincipality holds. 


\section{Concluding remarks}
\subsection{The complexity barrier}
We have already remarked in Section 2.2 that a flag algebra calculus-based approach cannot at present hope to give exact Tur\'an density results for 3-graphs on $7$ or more vertices. In this subsection, we shall consider some problems for small 3-graphs that we believe are still intractable, at least using the flag algebra method.

In contrast to the situation for graphs, we do not expect stability in general in extremal 3-graph theory. Indeed, we saw in Section 3.2 that if the conjectures of Tur\'an and S\'os are true then the Tur\'an problems for $K_4$ and $K_5$ are unstable. In fact generally the $K_t$ problem is conjectured to be unstable, non-isomorphic families of constructions having been given by Keevash and Mubayi~\cite{Keevash}. We mentioned another example of conjectured instability in Section 3.1 when we considered the Erd\H os-S\'os conjecture on odd cycles in link graphs and added many new constructions to the two given by Frankl and F\"uredi~\cite{FranklFuredi}.

Whatever the method used, unstable problems tend of course to be more difficult to handle than stable ones, and the flag algebra calculus is no exception to this trend. The bounds yielded by \flagmatic\ on the three problems mentioned above are 
\begin{align*}
5/9&\leq\pi(K_4) < 0.561666,\\ 
3/4 &\leq \pi(K_5)  <  0.769533 \textrm{ and}\\
1/4&\leq \pi(\Text{odd cycles in link graph}) < 0.258295
\end{align*} 
respectively, and we do not believe that these can be made tight even by an increase in computational firepower. A heuristic justification for our pessimism is as follows: the flag algebra calculus obtains bounds by considering how flags can intersect with each other; this information is then used to give inequalities which must be satisfied by the admissible subgraph densities. In an unstable problem however, several very different global intersection structures are possible, and what is a correct, sharp subgraph density inequality in one structure may well be false in another. Indeed some admissible subgraphs may be present in one extremal configuration with strictly positive density, but absent in another. As remarked in Section 2.2, a hypothetical tight flag algebra bound would have to be tight on all such subgraphs simultaneously; this seems a rather unlikely situation to hope for. In this sense, unstable problems appear to be beyond the scope of the flag algebra calculus method at present.

Another hurdle we have to face is that of stable problems with `complex' extremal configurations. Let us define more precisely what we mean by this. Recall the definition of \emph{blow-up} and \emph{iterated blow-up} introduced in Section 2.1. Currently all known stable extremal configurations for 3-graphs consist of blow-ups of some (possibly degenerate) 3-graphs. Frankl and F\"uredi gave however an iterated blow-up construction for the $K_4^-$ problem which is conjectured to be best possible. Since Frankl and F\"uredi's paper, Mubayi and R\"odl~\cite{MubayiRodl} (for the $C_5$ problem) and Bollob\'as, Leader and Malvenuto~\cite{BollobasLeaderMalvenuto} (for the $J_4$ problem) have both given us instances of the Tur\'an density problem where an iterated blow-up construction is conjectured to be best possible. To these let us add a fourth:

\begin{conjecture} \label{k4-c5conj}
\[\pi(K_4^-,C_5)=1/4.\]
\end{conjecture}
The lower bound in Conjecture~\ref{k4-c5conj} is attained for example by a balanced iterated blow-up of the 3-edge $([3], \{123\})$, or by a balanced iterated blow-up of $H_7$. To give motivation for our conjecture, let us note that we can get the following bounds on $\pi(K_4^-,C_5)$:
\begin{proposition} \label{k4-c5}
\[1/4 \leq \pi(K_4^-,C_5) < 0.251073 \]
\end{proposition}
\begin{proof}
The upper bound is from a flag algebra calculation using \flagmatic\ (see Section 2.4 for how to obtain a certificate). The lower bound is from an iterated blow-up of the 3-edge---this has bipartite links, hence is $K_4^-$-free. Moreover 5-sets of vertices are easily seen to span $0,1,2,3$ or $4$ edges, which is not sufficient for a copy of $C_5$ to appear as a subgraph. 
\end{proof}

Deferring our discussion of the limits of the flag algebra calculus for the moment, let us state why one should reasonably expect iterated blowup constructions to be the best possible for the $K_4^-=J_3$ and the $J_4$ problem, or indeed for the $J_t$ problem in general. (Why it should crop up in problems involving $C_5$ seems a little more mysterious.)

Suppose we have a non-degenerate 3-graph $H$ on $l$ vertices which is $J_t$-free. Then any iterated blow-up of $H$ will be $J_t$-free. Indeed, let $G$ be an iterated blow-up of $H$. Let $x \in V(G)$ and let us show its link graph is $K_{t}^{(2)}$-free. Consider a $t$-set of vertices $\{a_1, a_2, \dots, a_t\}$ in $G_v$. If all of the $a_i$ lie in the same level 1 part of $G$ as $v$, we can drop down to a lower level of the iterated construction, so we may assume without loss of generality $v \in A_0$ and $a_1 \in A_1$, where $A_0, A_1$ are two distinct level 1 parts. As $H$ was non-degenerate, there are no edges of type $A_0A_0A_0$, $A_1A_1A_1$, $A_0A_0A_1$ or $A_0A_1A_1$ in $G$. Thus for the purpose of finding a copy of $K_t^{(2)}$ in $G_v$ we may assume that $v, a_1, a_2, \dots, a_t$ all lie in different level 1 parts $A_0, A_1, A_2, \dots, A_{t}$ of $G$. But then the subgraph of $G$ induced by $v, a_1, a_2, \dots, a_t$ is isomorphic to a subgraph of $H$, which by hypothesis has $K_t^{(2)}$-free link graphs. Thus $G$ has $K_t^{(2)}$-free link graphs and is $J_t$-free as claimed. It follows from this that for the $J_t$ problem blow-up constructions cannot be best possible. (Note that blowing up a 3-graph containing a degenerate edge trivially gives a copy of $J_t$, so that our argument above does indeed cover all possible cases.)

Iterated blow-up constructions are therefore far from pathological, and one should expect them to crop up frequently in extremal 3-graph theory. Their structure is however much harder to grasp than that of their blow-up relatives. For example, the blow-up of a 3-graph $H$ (with no degenerate edge of the form $vvv$) will always be $\size{V(H)}$-partite. In contrast, for any $N \in \mathbb{N}$ sufficiently large (nontrivial) iterated blow-ups will fail to be $N$-partite: the level 1 edges force at least two parts, then looking into one of the parts, the level 2 edges force at least one more part, then looking into one of the subparts, the level 3 edges force at least one more part, and so on. This is one reason we would not expect the structure of iterated blow-up configurations to be properly captured by the flag algebra calculus.

\begin{proposition}\label{itblowupbounds}
\begin{align*}
2/7 \leq \pi(K_4-) &\leq 0.286889,\\
2\sqrt{3} -3 \leq \pi(C_5) &\leq 0.468287, \ \text{and} \\
1/2  \leq \pi(J_4) &\leq 0.504081.
\end{align*}
\end{proposition}
\begin{proof}
The upper bounds are from three flag algebra calculations using \flagmatic\ (see Section 2.4 for how to obtain a certificate). Lower bounds from (respectively) a balanced iterated blow-up of $H_6$~\cite{FranklFuredi}, a blow-up of $([2], \{112\})$ with $\size{A_1} \approx \sqrt{3} \size{A_2}$ and the construction iterated inside $A_2$~\cite{MubayiRodl}, and a balanced iterated blow-up of the complement of the Fano plane~\cite{BollobasLeaderMalvenuto}.
\end{proof}

We do not believe that the above three upper bounds can be made tight by a purely flag algebra calculus approach, and similarly we do not expect Conjecture \ref{k4-c5conj} to be resolved this way either.

Let us give here some heuristic justification for our pessimism regarding these bounds, beyond the mere fact that they fail to be tight. Given a nontrivial graph $H$ on $t$ vertices and an integer $k$, the number of subgraphs of order $k$ with strictly positive density in large blow-ups of $H$ will grow polynomially in $k$. Indeed let $H^+$ be a blow-up of $H$ such that for each vertex $x$ of $H$, a strictly positive proportion of the vertices of $H^+$ lie in the part $A_x$ associated with $x$. Given $H^+$'s $t$-partite structure, a $k$-subgraph of $H^+$ is entirely determined by the number of vertices it meets in each of $H^+$'s $t$ parts. Thus, up to a constant order correction factor, we expect the number of $k$-subgraphs of $H^+$ to be roughly $ \binom{k+t-1}{t-1}$.

By contrast, the number of subgraphs of order $k$ found in an iterated blow-up of $H$ will be superpolynomial. Indeed, suppose for simplicity's sake that $H^{\oplus}_n$ is a large balanced iterated blow-up of $H$ of order $n$, and let $f_n(k)$ be the number of $k$-subgraphs of $H^{\oplus}_n$. It is straightforward that for any fixed $k$, $f_n(k)$ converges to some number $f(k)$. 
Now pick some integer $K$. The value of $f(K)$ is then, up to a constant order correction factor,
\begin{equation}
\sum_{x_1+x_2+ \cdots +x_t=K} \prod_i f(x_i), \label{hplus}
\end{equation}
where the sum is taken over all partitions of $K$ into $s$ nonnegative integers $x_1, x_2, \dots, x_s$. Since $H^{\oplus}_n$ contains a large blow-up of $H$ as a subgraph, we know $f(k)$ has to grow at least at polynomial rate in $k$.  The estimate (\ref{hplus}) then implies $f(k)$ grows in fact faster than any polynomial.

This superpolynomial growth rate is an objective measure of the fact that iterated blow-ups are significantly more `complex' as 3-graph configurations than blow-up constructions. Computationally speaking, it is very bad news for an approach based on the flag algebra calculus. As we remarked in Section 2.1, if we obtain the correct upper bound on a Tur\'an density problem, the flag algebra bound must be tight on all subgraphs that appear with strictly positive density in an extremal construction, whereas some slack is expected for the rest of the admissible subgraphs. In this sense iterated blow-up constructions require us to prove far more delicate inequalities than mere blow-up constructions---the far richer subgraph structure of iterated blow-ups leaving us with much less room to spare in our optimisation, making our task significantly harder. We therefore expect that most attempts to attack problems admitting iterated blowups as extremal constructions with \flagmatic\ will run into the limits set by the SDP solver and fail to get tight bounds.

\subsection{Further open problems}
The most obvious challenge our discussion above leaves open is the following. Say that a Tur\'an problem is \emph{simple} if the number of subgraphs of order $k$ which can occur with density bounded below by some $\varepsilon>0$ in an extremal configuration grows polynomially in $k$, and that a Tur\'an problem is \emph{complex} otherwise.

\begin{question}
Can we obtain an exact Tur\'an density result for a complex problem?
\end{question}

In Section 3.2 we proved a number of results in the extremal theory of 3-graphs with independent neighbourhoods. As the extremal construction for $F_{3,2}$ is $K_4$-free, it is easy to see that $\pi(K_t, F_{3,2})=4/9$ for all $t \geq 4$. Having considered both the $J_t$ (complete graphs in links) and the odd cycle in links problem, the most natural question to ask next is perhaps: what happens if instead of forbidding all odd cycles we only forbid odd cycles of a given length in the link graphs? For example:
\begin{question}~\label{longoddcycle}
Is $\pi(F_{3,2}, \Text{ odd cycle of length at least 5 in link})=1/4$?
\end{question}
and
\begin{question}~\label{shortoddcycle}
Is $\pi(F_{3,2}, \Text{ odd cycle of length at most 5 in link})=1/4$?
\end{question}
Note that if a vertex in a 3-graph $G$ has a triangle in its link graph, then for any odd length $l\geq 3$, sufficiently large blow-ups of $G$ will have link graphs containing odd cycles of length $l$; were it not for the nature of Construction~\ref{c1}, this would suggest the answer to Question~\ref{longoddcycle}  is `Yes'. Also, Theorem~\ref{k4-f32} tells us the answer to Question~\ref{shortoddcycle} is `No' if we replace $7$ by $5$ (since $\pi(K_4^-, F_{3,2})=5/18$), making the question more open-ended than the upper bounds we are able to obtain on the problem using \flagmatic\ suggest.

\subsection{Summary of results and constructions}
We set in the table below the constructions and \flagmatic\ bounds for the Tur\'an density problems discussed in the paper.

\begin{center}
{\small
\begin{tabular}{  p{3.5cm}  p{1.9cm}  p{2cm} p{5.5cm}}
\hline
Forbidden graphs & Lower bound for $\pi$ &  Upper bound for $\pi$ 	& (Conjectured) Extremal configuration(s) \\
\hline
$K_4^-, \ C_5, \ F_{3,2}$ & $12/49$& $12/49$  &  Blow-up of $H_7$. \\   
$K_4^-, \ C_5$ & $1/4$ & $0.251073$  &  Iterated blowup of a 3-edge. \\ 
$F_{3,2}$, odd cycle in links & $1/4$ & $0.255886$ & Geometric~\cite{FranklFuredi}; see Construction~1 in Section 3.2. \\
odd cycle in links & $1/4$ & $0.258295$ & Many; see Section 3.2. \\  
$K_4^-, \ F_{3,2}$ & $5/18$ & $5/18$ &  Blowup of $H_6$~\cite{FranklFuredi}. \\ 
$K_4^-$ & $2/7$ & $0.286889$ &  Iterated blowup of $H_6$~\cite{FranklFuredi}. \\ 
$J_4, \ F_{3,2}$ & $3/8$ & $3/8$ & Blow-up of $K_4$. \\
$J_5, \ F_{3,2}$ & $3/8$ & $3/8$ & Blow-up of $K_4$. \\ 
$F_{3,2}$, \ induced $K_4^-$ & $3/8$ & $3/8$ & Blow-up of $K_4$. \\ 
$F_{3,2}$ & $4/9$ & $4/9$~\cite{FurediPikhurkoSimonovits} & Bipartition of the vertex set into two parts $A$ and $B$ with $\size{A} \approx 2 \size{B}$, all edges of type $AAB$~\cite{FurediPikhurkoSimonovits}. \\ 
$J_4, \ K_4$ & $2/5$ & $0.479371$ & \\ 
$J_4$ & $1/2$ & $0.504081$ & Iterated blowup of the complement of the Fano plane~\cite{BollobasLeaderMalvenuto}. \\ 
$C_5$ & $2\sqrt{3}-3$ & $0.468287$ &  Bipartition of the vertex set into two parts $A$ and $B$ with $\size{A} \approx \sqrt{3} \size{B}$, all edges of type $AAB$, then iterate inside $B$~\cite{MubayiRodl}. \\
$K_4$, \ induced $G_1$ & $5/9$ & $5/9$~\cite{Razborov2} & Tur\'an's construction. \\ 
$K_4$ & $5/9$ &$0.561666$~\cite{Razborov2} &  Many; see~\cite{Brown, FonderFlaass, Frohmader, Kostochka}. \\
$K_5$, \ 5-set spanning 8 edges & $3/4$ &$3/4$ & Complete bipartite graph. \\
$K_5$ & $3/4$ & $0.769533$~\cite{BaberThesis} & Many; see~\cite{Sidorenko}. \\
\hline
\end{tabular}
}
\end{center}

\section*{Acknowledgements}

The authors would like to thank Peter Keevash and Mark Walters for their helpful comments and discussions.

\end{document}